\documentclass[12pt]{article}

\usepackage{setspace}
\usepackage{amsthm, amssymb, amstext}
\usepackage[fleqn]{amsmath}
\usepackage{latexsym}
\usepackage[dvips]{graphicx}
\usepackage{comment}
\usepackage{hyperref}
\usepackage{mathtools}
\usepackage{enumerate} 

\usepackage{todonotes}
\usepackage{comment}

\newtheorem{theorem}{Theorem}

\newtheorem{corollary}{Corollary}

\theoremstyle{definition}

\usepackage{amsthm}

\title{Reconfiguring 10-colourings of planar graphs}
\author{Carl Feghali \\ Department of Informatics, \\ University of Bergen, \\ Bergen, Norway\\
{\small \texttt{carl.feghali@uib.no}}}
\date{}

\begin{document}
\maketitle

\begin{abstract}
Let $k \geq 1$ be an integer. The reconfiguration graph $R_k(G)$ of the $k$-colourings
of a graph~$G$ has as vertex set the set of all possible $k$-colourings
of $G$ and two colourings are adjacent if they differ on exactly
one vertex.  

A conjecture of Cereceda from 2007 asserts that for every integer $\ell \geq k + 2$ and $k$-degenerate graph $G$ on $n$ vertices, $R_{\ell}(G)$ has diameter $O(n^2)$. The conjecture has been verified only when $\ell \geq 2k + 1$. We give a simple proof that if $G$ is a planar graph on $n$ vertices, then $R_{10}(G)$ has diameter at most $n^2$. Since planar graphs are $5$-degenerate, this affirms Cereceda's conjecture for planar graphs in the case $\ell = 2k$.  \end{abstract}

Let $k \geq 1$ be an integer. The reconfiguration graph $R_k(G)$ of the $k$-colourings
of a graph~$G$ has as vertex set the set of all possible $k$-colourings
of $G$ and two colourings are adjacent if they differ on the colour of exactly
one vertex of $G$.  A \emph{list assignment} of a graph  is a function $L$ that assigns to each vertex $v$ a list $L(v)$ of colours. The graph $G$ is \emph{$L$-colourable} if it has a proper colouring $f$ such that $f(v) \in L(v)$ for each vertex $v$ of $G$.  

For a positive integer $d$, a graph $G$ is \emph{$d$-degenerate} if every subgraph of $G$ contains a vertex of degree at most $d$. Expressed in another way, $G$ is $d$-degenerate if there there exists an ordering $v_1, \dots, v_n$ of the vertices in $G$ such that each $v_i$ has at most $d$ neighbours $v_j$ with $j < i$. 

Reconfiguration problems have received much attention in the past decade; we refer the reader to the surveys by van den Heuvel \cite{He13} and Nishimura \cite{nishimura}.  

In this note, we are concerned with a conjecture of  Cereceda \cite{luisthesis} from 2007 which asserts that for every integer $\ell \geq k + 2$ and $k$-degenerate graph $G$ on $n$ vertices, $R_{\ell}(G)$ has diameter $O(n^2)$. Cereceda \cite{luisthesis} verified the conjecture  whenever $\ell \geq 2k + 1$ but the conjecture remains open for every other value $2k \geq \ell \geq k + 2$. It is also known to hold for graphs of bounded tree-width \cite{bonamy13} (the claimed shorter proof in \cite{feghalipaths} yields instead a bound of $O(tn^2)$ on the diameter, where $t$ is the tree-width of the graph under consideration) as well as $(\Delta - 1)$-degenerate graphs \cite{brooks}, where $\Delta$ is the maximum degree of the graph under consideration. Our aim in this note is to address the conjecture for planar graphs in the following theorem. 

\begin{theorem}\label{main:thm}
For every planar graph $G$ on $n$ vertices, $R_{10}(G)$ has diameter at most~$n^2$. 
\end{theorem}

Since planar graphs are $5$-degenerate, Theorem \ref{main:thm} affirms Cereceda's conjecture for planar graphs in the case $\ell = 2k$. In all other cases, some partial results are known. Given a planar graph $G$ on $n$ vertices, it is shown in \cite{bousquet11} that $R_\ell(G)$ has diameter $O(n^c)$ for each $\ell \geq 8$ and some (large) positive constant $c$ (see \cite{feghali} for a short proof of this result with a weaker bound on $c$) while in \cite{eiben} it is shown that $R_7(G)$ has diameter $2^{O(\sqrt{n})}$. 

Let us note that the novelty of our approach lies, in some sense, on a new trick that essentially reformulates the reconfiguration problem as a list colouring problem. In particular, Theorem \ref{main:thm} will follow as a corollary from the following special case of a famous theorem due to Thomassen \cite{thomassen1}.

\begin{theorem}\label{thm:planar}
Let $G$ be a planar graph, and let $v$ be a vertex of $G$. Suppose  that $L(u)$ is a list of one colour if $u = v$ and a list of at least five colours if $u \in V(G) - \{v\}$. Then $G$ is $L$-colourable.   
\end{theorem}

\begin{proof}[Proof of Theorem \ref{main:thm}]
Since $G$ is $5$-degenerate, we can order the vertices of $G$ as $v_1, \dots, v_n$ such that each $v_i$ has at most five neighbours $v_j$ with $j < i$.  

Let $\alpha$ and $\beta$ be $10$-colourings of $G$, and
let $h$ be the lowest index such that $\alpha(v_h) \not = \beta(v_h)$. Starting from $\alpha$, we shall describe a sequence of recolourings such that
\begin{itemize}
\item for $i < h$, $v_i$ is not recoloured,
\item for $i > h$, $v_i$ is recoloured at most once, and
\item $v_h$ is recoloured with colour $\beta(v_h)$. 
\end{itemize}
By repeatedly using such sequences, we can recolour $\alpha$ to $\beta$ by at most $n$ recolourings per vertex and the theorem follows. 

To describe the sequence, let $H$ be the graph induced by $S = \{v_h, \dots, v_n\}$. We start by finding a list assignment $L$ of $H$  as follows:
\begin{itemize}
\item $L(v_h) = \{\beta(v_h)\}$, and
\item for $i > h$, $L(v_i) = \{1, \dots, 10\} \setminus \{\alpha(v_j): (v_i, v_j) \in E(G), j < i \}$. 
\end{itemize} 
Applying Theorem \ref{thm:planar}, we obtain an $L$-colouring $f$ of $H$.  We then simply recolour $v_k$ with $f(v_k)$  starting with $v_n$ and working backwards through $S$. (It is possible that $f(v_k) = \alpha(v_k)$ in which case the colour of $v_k$ is unchanged.) Each colouring obtained is proper since $v_n$ has no neighbours coloured $f(v_n)$ and when a vertex $v_k$, $k < n$, is recoloured, its neighbours $v_j$ with $j < k$ do not have colour $f(v_k)$ by definition  of $L(v_k)$ nor do its other neighbours $v_j$ with $j > k$ since $f$ is a proper colouring. Given that, at the end of the sequence, $v_h$ is recoloured to colour $\beta(v_h)$, this completes the proof.  
\end{proof}

It is not difficult to prove the following theorem using the same approach as in the proof of Theorem \ref{main:thm}. 

\begin{theorem}\label{thm:cor}
Let $k$ and $\ell$ be positive integers, let $G$ be a $k$-degenerate graph on $n$ vertices, and let $v$ be a vertex of $G$. Suppose that $L(u)$ is a list of one colour if $u = v$ and a list of at least $\ell$ colours if $u \in V(G) - \{v\}$. If $G$ is $L$-colourable, then $R_{k + \ell}(G)$ has diameter at most $n^2$. 
\end{theorem}

We state two of possibly other consequences of Theorem \ref{thm:cor}. 

\begin{corollary}
Let $G$ be a planar graph on $n$ vertices and of girth 5. Then $R_6(G)$ has diameter at most $n^2$. 
\end{corollary}

\begin{proof}
Since planar graphs of girth 5 are 3-degenerate, the result is immediate from Theorem \ref{thm:cor} combined with Theorem 2.1 in \cite{thomassen2}. 
\end{proof}

\begin{corollary}\label{corollary1}
Let $k$ be a positive integer, and let $G$ be a $k$-degenerate graph on $n$ vertices.  If $k + 1$ is prime, then $R_{2k + 1}(G)$ has diameter at most $n^2$. 
\end{corollary}

\begin{proof}
Combine Theorem \ref{thm:cor} with Theorem 6 in \cite{dvorak}. 
\end{proof}

\section*{Acknowledgements}

The author is grateful to Louis Esperet for pointing out Corollary \ref{corollary1}. This work is supported by the Research Council of Norway via the project CLASSIS. It was conducted while the author was visiting the Department of Applied Mathematics of the Faculty of Mathematics and Physics at Charles University.

 \bibliography{bibliography}{}
\bibliographystyle{abbrv}
 
\end{document}